\documentclass[12pt,a4paper]{amsart}
\usepackage[T1]{fontenc}
\usepackage{amssymb}
\usepackage{amsthm}

\newtheorem{theorem}{Theorem}
\newtheorem{lemma}{Lemma}

\newtheorem{remark}{Remark}

\title[Almost everywhere divergence of Ces\`aro means \ldots]{Almost everywhere divergence of Ces\`aro means of subsequences of Walsh--Paley Fourier partial sums}
\date{}

\author[I. Blahota]{Istv\'an Blahota}
\address{I. Blahota, Institute of Mathematics and Computer Sciences, University of Ny\'\i regyh\'aza, H-4400 Ny\'\i regyh\'aza, S\'ost\'oi street 31/b, Hungary}
\email{blahota.istvan@nye.hu}

\author[G. G\'at]{Gy\"orgy G\'at}
\address{G. G\'at, Institute of Mathematics, University of Debrecen, Pf. 400, 4002 Debrecen, Hungary}
\email{gat.gyorgy@science.unideb.hu}

\thanks{The second author was supported  by the University of Debrecen Program for Scientific Publication.}

\begin{document}

\begin{abstract}
	We prove an almost everywhere divergence theorem for Ces\`aro means of subsequences of partial sums of Walsh--Fourier series. More precisely, we show that	there exist a strictly increasing sequence of positive integers $(a_n)$ and a function $f\in L^1(G)$, where $G$ denotes the dyadic group, such that
	\[
	\left(
	\frac1N\sum_{n=1}^N S_{a_n}f(x)
	\right)_{N=1}^{\infty}
	\]
	does not converge for almost every $x\in G$.
	
	The result is motivated by the corresponding trigonometric problem. 
	
	A problem going back to Zalcwasser \cite{Zalcwasser1936} and remaining open for almost
	ninety years was recently solved in the trigonometric setting by G\'at \cite{GatTrigSubseqDivergence}, who proved almost everywhere divergence for	suitable subsequential arithmetic means of partial sums.
	
	In the Walsh--Paley setting the situation is more delicate. The dyadic structure gives positive convergence results; for instance, the partial	sums $S_{2^m}f$ converge almost everywhere to $f$ for every $f\in L^1(G)$. Our theorem shows that, despite this stabilizing dyadic structure, suitably chosen arithmetic means of Walsh--Fourier partial sums may still diverge almost everywhere.
\end{abstract}
	
\maketitle

\section{Introduction}

A central problem in Fourier analysis is the reconstruction of a function from the partial sums of its Fourier series. In the trigonometric case, the classical theorem of Carleson, later extended by Hunt, asserts the almost everywhere convergence of the partial sums for functions in $L^p$, $p>1$ \cite{Carleson1966,Hunt1968}. On the other hand, Kolmogorov's examples show
that this assertion cannot be extended to the whole space $L^1$: there exists an integrable function whose trigonometric Fourier series diverges almost everywhere, and even an everywhere divergent Fourier series \cite{Kolmogorov1923,Kolmogorov1926}.

The situation is even more delicate if one considers only a subsequence of the partial sums. In the trigonometric setting this question has a very strong negative aspect. Gosselin proved that for every subsequence $(n_j)$ of the positive integers there exists an $f\in L^1(\mathbb T)$ such that
\[
	\sup_{j\in\mathbb{P}} |S_{n_j}f|=+\infty
\]
almost everywhere. Totik later strengthened this result by obtaining divergence everywhere \cite{Gosselin1958,Totik1982}. Thus, in the trigonometric case, no prescribed subsequence of partial sums can guarantee pointwise convergence for all integrable functions.

Since ordinary convergence of partial sums may fail in $L^1$, it is natural to consider arithmetic means. The classical theorems of Fej\'er and Lebesgue show that the Ces\`aro means
\[
\sigma_N f=\frac1N\sum_{k=1}^N S_kf
\]
converge almost everywhere to $f$ for every $f\in L^1$ in the trigonometric case \cite{Fejer1904,Lebesgue1905}. This leads to a more subtle question, going back to Zalcwasser, namely what can be said about arithmetic means formed from a subsequence of partial sums,
\[
	\frac1N\sum_{j=1}^N S_{n_j}f .
\]

Zalcwasser proved in 1936 that for $n_j=j^2$ these means converge almost everywhere to $f$ for every $f\in L^1(\mathbb T)$ \cite{Zalcwasser1936}. Salem mentioned related extensions to $j^3$ and $j^4$, while Belinsky later constructed sequences of substantially faster growth for which the same almost everywhere convergence still holds \cite{Salem1955,Belinsky1997}. Thus Zalcwasser's result naturally led to the problem of whether such subsequential arithmetic means converge almost everywhere for every integrable function under more general assumptions on the subsequence.

For continuous functions and uniform convergence, the corresponding trigonometric problem is essentially understood. In the convex case, the condition
\[
	\sup_{j\in\mathbb{P}} j^{-1/2}\log n_j<+\infty
\]
is necessary and sufficient for the uniform convergence of the corresponding arithmetic means for every continuous function \cite{Carleson1983}; see also \cite{Belinsky1984,KahaneKatznelson1983,Salem1955}. For $L^1$ functions and almost everywhere convergence, however, the picture is more complicated. The problem, going back to Zalcwasser and remaining open for almost ninety years, was recently solved in the negative by G\'at: there exist a subsequence $(k_j)$ and an integrable function $f$ such that
\[
\frac1N\sum_{j=1}^N S_{k_j}f
\]
does not converge to $f$ almost everywhere \cite{GatTrigSubseqDivergence}. Moreover, for convex sequences the corresponding super $(C,1)$-summability property is characterized by almost lacunarity \cite{GatTrigSubseqDivergence}. The positive lacunary convergence result used in this context goes back to \cite{Gat2019}.

The Walsh--Paley system exhibits a substantially different behaviour. The Walsh analogue of Carleson's theorem was proved by Billard, while Fine proved the Walsh--Fej\'er--Lebesgue theorem: for every $f\in L^1(G)$ the Walsh--Ces\`aro means converge almost everywhere to $f$ \cite{Billard1967,Fine}. This difference is especially visible for subsequences of partial sums. The partial sums
\[
	S_{2^m}f
\]
are the conditional expectations of $f$ with respect to the 
\[
	\mathcal F_m:=\sigma\{I_m(x):x\in G\},
\]
generated by the dyadic intervals of order $m$. Hence, by the martingale convergence theorem, for every $f\in L^1(G)$,
\[
	S_{2^m}f(x)\to f(x)\quad \text{for a.e. }x\in G.
\]
Thus the Walsh--Paley system has natural subsequences along which the partial sums themselves converge almost everywhere for every integrable function. This is in sharp contrast with the trigonometric theorem of Gosselin and Totik.

A more general positive result is formulated in terms of the dyadic variation of the index. It is known that, for every increasing sequence $(a(n))$ satisfying
\[
	\sup_{n\in\mathbb{P}} V(a(n))<+\infty,
\]
one has
\[
	S_{a(n)}f(x)\to f(x)\quad \text{for a.e. }x
\]
for every $f\in L^1(G)$. The bounded-variation condition is not necessary: Konyagin answered a question of Balashov negatively by constructing a sequence $(a(n))$ with unbounded dyadic variation for which the above almost everywhere convergence still holds for every integrable function \cite{Konyagin1993}. These results show that the Walsh--Paley theory is not a direct translation of the trigonometric one; the binary structure may force convergence along
subsequences which have no comparable trigonometric analogue.

Arithmetic means of Walsh--Fourier partial sums along subsequences have also been studied. G\'at proved almost everywhere convergence results for Ces\`aro and logarithmic means of subsequences of Walsh--Fourier partial sums, including lacunary subsequences \cite{Gat2010}. Later it was shown that if the sequence $a$ satisfies
\[
	a(n+1)\ge\left(1+\frac1{(n+1)^\delta}\right)a(n),\qquad 0<\delta<\frac12,
\]
then
\[
	\frac1N\sum_{n=1}^N S_{a(n)}f(x)\to f(x)\quad \text{for a.e. }x
\]
for every $f\in L^1(G)$ \cite{Gat2019Walsh}. See also \cite{Gat2019}.

The purpose of the present paper is to complement these positive Walsh--Paley results by a divergence theorem for Ces\`aro means of a suitably chosen subsequence of Walsh--Fourier partial sums. We prove that there exist a strictly increasing sequence of positive integers $(a_n)$ and a function $f\in L^1(G)$ such that
\[
	\left(\frac1N\sum_{n=1}^N S_{a_n}f(x)\right)_{N=1}^{\infty}
\]
diverges for almost every $x\in G$. This result is significant precisely because of the stabilizing dyadic structure of the Walsh--Paley system. Unlike in the trigonometric case, many subsequences, for instance $(2^m)$ and more generally bounded-variation subsequences, yield almost everywhere convergence of the partial sums themselves for every integrable function. The theorem therefore shows that, despite these positive dyadic phenomena, Ces\`aro means of carefully selected Walsh--Fourier partial sums may still diverge almost everywhere.

\section{Notation and auxiliary facts}

Let $\mathbb P$ be the set of positive integers and let
$\mathbb N:=\mathbb P\cup\{0\}$. Let
\[
	G:=\prod_{k=0}^{\infty}\mathbb Z_2
\]
be the dyadic group, endowed with the product topology and the normalized Haar measure $\mu$. Addition on $G$ is coordinate-wise modulo $2$. For $x=(x_0,x_1,\ldots)\in G$, the dyadic interval is
\[
	I_0(x):=G,\qquad I_n(x):=\{y\in G:y=(x_0,\ldots,x_{n-1},y_n,y_{n+1},\ldots)\},
\]
where $n\in\mathbb P$. We write $I_n:=I_n(0)$.

The Rademacher functions are
\[
	r_n(x):=(-1)^{x_n}\qquad (x\in G, n\in\mathbb N).
\]
If
\[
	n=\sum_{k=0}^{\infty} n_k2^k,\qquad n_k\in\{0,1\},
\]
is the binary expansion of $n$, then the Walsh--Paley functions are
\[
	w_0(x):=1,\qquad w_n(x):=\prod_{k=0}^{\infty} r_k^{n_k}(x)=(-1)^{\sum_{k=0}^{|n|} n_kx_k}.
\]
The order and variation of $n$ are
\[
	|0|:=0,\textrm{ if } n\in\mathbb{P}\textrm{, then } |n|:=\max\{j\in\mathbb N:n_j\ne0\},\qquad V(n):=\sum_{k=1}^{\infty}|n_k-n_{k-1}|+n_0.
\]
For $f\in L^1(G)$, the Walsh--Fourier coefficients, partial sums and Dirichlet
kernels are
\[
	\widehat f(j):=\int_G f(x)w_j(x)\,d\mu(x),\qquad S_k(f):=\sum_{j=0}^{k-1}\widehat f(j)w_j,\qquad D_n:=\sum_{k=0}^{n-1}w_k.
\]
The Ces\`aro means and kernels are denoted by
\[
	\sigma_n(f):=\frac1n\sum_{k=1}^{n}S_k(f),\qquad K_n:=\frac1n\sum_{k=1}^{n}D_k.
\]

It is known that 
\[
	\sigma_n(f;x)=\int_G f(x+t)K_n(t)\,d\mu(t)
\]
and
\begin{equation}\label{K2np}
	K_{2^n}(x)\geq 0
\end{equation}
for every $n\in\mathbb{N}$ and $x\in G.$

\begin{theorem}[\cite{Fine}]\label{Fine_th}
If
$f\in L^1(G)$, then
\[
	\sigma_n(f;x)\to f(x) \quad \text{for a.e. } x\in G .
\]
\end{theorem}

\begin{lemma}[\cite{SWSP}]\label{V8}
If $n\in\mathbb N$, then
\[
	\frac{V(n)}8\le \|D_n\|_1\le V(n).
\]
\end{lemma}

\begin{lemma}[\cite{Paley}]\label{paley}
For every $n\in\mathbb N$ and $x\in G$,
\[
	D_{2^n}(x)=
	\begin{cases}
		2^n, & x\in I_n,\\
		0, & x\notin I_n.
	\end{cases}
\]
\end{lemma}
Let $L$ be an odd positive integer and let $A\in\mathbb N$, $L<A$. Define
\[
	n_{(A,L)}:=\sum_{s=0}^{(L-1)/2}2^{A-2s-1}=2^{A-L}+2^{A-L+2}+\cdots+2^{A-3}+2^{A-1}.
\]
Let $e_j=(0,\dots,0,1,0,\dots)\in G$ denote the element whose only non-zero coordinate is the $j$-th one, and put
\[
	U_{A,L}:=\left\{u\in G:u=\sum_{j=0}^{A-L-1}e_j u_j\right\},\qquad\widetilde U_{A,L}:=\left\{u\in G:u=\sum_{j=A-L}^{A-1}e_j u_j\right\}.
\]
Write $\widetilde U_{A,L}=\{v_0,\ldots,v_{2^L-1}\}$. Define
\[
	g_{A,L}(x):=\sum_{u\in U_{A,L}}\operatorname{sgn}D_{n_{(A,L)}}(x+u),
\]
\[
	R_{k,A,L}(x):=r_{A+k}(x)g_{A,L}(x+v_k)\qquad (0\le k<2^L),
\]
and
\[
	Q_{A,L}:=\prod_{k=0}^{2^L-1}(1+R_{k,A,L}).
\]

We use $c>0$ for an absolute constant, which may change from line to line.

\begin{lemma}[\cite{BGcd}]\label{bg10}	
If $x\in U_{A,L}$, then there exists an absolute constant $c>0$ such that
\[
	S_{n_{(A,L)}}(g_{A,L};x)\ge cL.
\]
\end{lemma}

\begin{lemma}[\cite{BGcd}]\label{bg12}
For every odd positive integer $L$ and every $A\in\mathbb N$, $L<A$,
\[
	Q_{A,L}\ge0,\qquad\|Q_{A,L}\|_1=1.
\]
\end{lemma}

\begin{lemma}[\cite{BGcd}]\label{bg13}
For every $x\in G$ there exists
\[
	k\in\{0,\ldots,2^L-1\}
\]
such that
\[
	S_{2^{A+k}+n_{(A,L)}}(Q_{A,L};x)-S_{2^{A+k}}(Q_{A,L};x)\ge cL.
\]
\end{lemma}

\begin{lemma}[\cite{BGcd}]\label{bg14}
Let $L_j$ be odd positive integers and let $A_j\in\mathbb N$. Assume that
\[
	A_j-L_j\to\infty,\qquad L_j\to\infty,\qquad A_{j+1}>A_j+2^{L_j},
\]
\[
	L_{j+1}\ge (j+1)2^{A_j+2^{L_j}+1}.
\]
Let $\alpha_1>0$ be fixed, and define
\[
	\alpha_{j+1}:=2^{-A_j-2^{L_j}-1}\qquad (j\ge1)
\]
and
\[
	f:=\sum_{j=1}^{\infty}\alpha_j Q_{A_j,L_j}.
\]
Then, for every $x\in G$ and every sufficiently large $j$, there exists
\[
	k=k(j,x)\in\{0,\ldots,2^{L_j}-1\}
\]
such that
\[
	S_{2^{A_j+k}+n_{(A_j,L_j)}}(f;x)-S_{2^{A_j+k}}(f;x)\ge cj.
\]
\end{lemma}

\section{Results}

	\begin{lemma}\label{A}
		Let the sequences $A_j,L_j$ and the function
		$f$ be chosen as in Lemma \ref{bg14}. Then $f\in L^1(G)$ and $f\ge 0$.
		
		For every $x\in G$ and every sufficiently large $j$ there exists
		\[
			k=k(j,x)\in\{0,\ldots,2^{L_j}-1\}
		\]
		such that
		\[
			S_{2^{A_j+k}+n_{(A_j,L_j)}}(f;x)-S_{2^{A_j+k}}(f;x)\ge c j.
		\]
		
		Moreover, there exists a set $E\subset G$ with $\mu(E)=0$ such that, for every $x\in G\setminus E$, if $q_j=q_j(x)$ is any sequence of powers of two satisfying
		\[
			1\le q_j\le 2^{A_j-L_j-1}
		\]
		and $q_j\to\infty$, then, for the same $k=k(j,x)$,
		\[
			\frac1{q_j}\sum_{\ell=2^{A_j+k}+n_{(A_j,L_j)}+1}^{2^{A_j+k}+n_{(A_j,L_j)}+q_j}S_\ell(f;x)\ge cj-o(j).
		\]
	\end{lemma}
	
	\begin{proof}
		We recall how the lower bound follows from Lemmas \ref{V8}, \ref{bg10}, \ref{bg13} and \ref{bg14}. Lemma \ref{V8} gives $\|D_n\|_1\ge c\,V(n)$. In Lemma \ref{bg10} this is applied to an integer whose binary digits alternate on a
		block of length $L$. Then $V(n)\ge cL$, so
		\[
			S_{n_{(A,L)}}(g_{A,L};x)\ge cL
		\]
		for every $x\in U_{A,L}$. Then Lemma \ref{bg13} gives, for every $x$, a suitable shift $k$, and Lemma \ref{bg14} transfers this estimate to $f$. So for every $x$ and all large $j$ there is $k=k(j,x)$ such that
		\[
			S_{2^{A_j+k}+n_{(A_j,L_j)}}(f;x)-S_{2^{A_j+k}}(f;x)\ge cj.
		\]
		
		By Lemma \ref{bg12} we have $Q_{A_j,L_j}\ge0$ and $\|Q_{A_j,L_j}\|_1=1$. Since $\alpha_j>0$ and $\sum_j\alpha_j<\infty$, we get $f\ge0$ and $f\in L^1(G)$.
		
		Since $2^{A_j+k}$ is a power of two, Lemma \ref{paley} gives $D_{2^{A_j+k}}(u)=2^{A_j+k}{\bf 1}_{I_{A_j+k}}(u)$.
		Hence
		\[
			S_{2^{A_j+k}}(f;x)=\int_G f(t)\,D_{2^{A_j+k}}(x+t)\,d\mu(t)\ge0,
		\]
		because $f\ge0$. Therefore, with
		\[
			M:=2^{A_j+k}+n_{(A_j,L_j)},
		\]
		we have $S_M(f;x)\ge cj$.
		
		The number $n_{(A_j,L_j)}$ is divisible by $2^{A_j-L_j}$, and $2^{A_j+k}$ is also
		divisible by $2^{A_j-L_j}$. So $M$ is divisible by $2^{A_j-L_j}$. If $1\le r\le q$ and $q\le 2^{A_j-L_j-1}$, then $r<2^{A_j-L_j}$. Hence the addition $M+r$ is carry-free in binary representation. Consequently,
		\[
			w_{M+\nu}=w_M w_\nu \qquad (0\le \nu<r,\ 1\le r\le q),
		\]
		and therefore
		\[
			S_{M+r}(f;x)=S_M(f;x)+w_M(x)\,S_r(fw_M;x)\qquad (1\le r\le q).
		\]
		
		Averaging over $r=1,\ldots,q$ gives
		\[
			\frac1q\sum_{r=1}^{q}S_{M+r}(f;x)=S_M(f;x)+w_M(x)\,\sigma_q(fw_M;x),
		\]
		where
		\[
			\sigma_q(g):=\frac1q\sum_{r=1}^{q}S_r(g).
		\]
		Hence
		\[
			\frac1q\sum_{r=1}^{q}S_{M+r}(f;x)\ge cj-\left|\sigma_q(fw_M;x)\right|.
		\]
		
		Since $q$ is a power of two, by inequality \eqref{K2np} the Walsh--Ces\`aro kernel $K_q$ is non-negative, so
		\[
			\left|\sigma_q(fw_M;x)\right|\le \sigma_q(|f|;x).
		\]
		Let $E$ be the exceptional null set in Theorem \ref{Fine_th} applied to $|f|$. Then, for every $x\in G\setminus E$,
		\[
			\sigma_q(|f|;x)\to |f(x)|\qquad (q\to\infty).
		\]
		Hence, for every $x\in G\setminus E$, the same convergence holds along every sequence $q_j=q_j(x)\to\infty$. In particular,
		\[
			\sigma_{q_j}(|f|;x)=o(j).
		\]
		Therefore the claimed estimate follows.
	\end{proof}
	
	For $0\le k<2^{L_j}$ define
	\[
		p_{k,j}:=2^{A_j-L_j-2^{L_j}-2+k},\qquad P_{j}:=\sum_{m=1}^{j}\sum_{\nu=0}^{2^{L_{m}}-1}p_{\nu,m},\qquad	P_{0}:=0.
	\]

	\begin{lemma}\label{B}
		Assume $A_j\ge L_j+2^{L_j}+2$ so that $p_{0,j}\ge1$.
		
		Then
		\[
			\sum_{\nu=0}^{k-1}p_{\nu,j}<p_{k,j}\qquad (1\le k<2^{L_j}).
		\]
		Moreover, if
		\[
			P_{j-1}<p_{0,j},
		\]
		then
		\[
			P_{j-1}+\sum_{\nu=0}^{k-1}p_{\nu,j}<p_{k,j}\qquad (0\le k<2^{L_j}).
		\]
		Consequently, with
		\[
			P_{k,j}:=\sum_{\nu=0}^{k}p_{\nu,j},\qquad N_{k,j}:=P_{j-1}+P_{k,j},
		\]
		we have
		\[
			\frac{p_{k,j}}{N_{k,j}}>\frac12\qquad (0\le k<2^{L_j}). 
		\]
	\end{lemma}
	
	\begin{proof}
		We have $p_{k,j}=2^k p_{0,j}$. Hence, for $1\le k<2^{L_j}$,
		\[
			\sum_{\nu=0}^{k-1}p_{\nu,j}=p_{0,j}\sum_{\nu=0}^{k-1}2^\nu=p_{0,j}(2^k-1)<	p_{0,j}2^k=p_{k,j}.
		\]
		
		Assume now that $P_{j-1}<p_{0,j}$. If $k=0$, then
		\[
			P_{j-1}+\sum_{\nu=0}^{-1}p_{\nu,j}=P_{j-1}<p_{0,j}=p_{k,j}.
		\]
		If $1\le k<2^{L_j}$, then
		\[
			P_{j-1}+\sum_{\nu=0}^{k-1}p_{\nu,j}<p_{0,j}+p_{0,j}(2^k-1)=p_{0,j}2^k=		p_{k,j}.
		\]
		Thus, for every $0\le k<2^{L_j}$,
		\[
			N_{k,j}=P_{j-1}+\sum_{\nu=0}^{k-1}p_{\nu,j}+p_{k,j}<2p_{k,j},
		\]
		and therefore
		\[
			\frac{p_{k,j}}{N_{k,j}}>\frac12.
		\]
	\end{proof}
	
	\begin{lemma}\label{C}
		Define
		\[
			I_{k,j}:=\left[2^{A_j+k}+n_{(A_j,L_j)}+1,\,	2^{A_j+k}+n_{(A_j,L_j)}+p_{k,j}\right]\cap\mathbb N.
		\]
		For fixed $j$, the intervals $I_{k,j}$ are pairwise disjoint. By choosing $A_{j+1}$ sufficiently large, the intervals belonging to different levels $j$ can also be made pairwise disjoint and increasingly ordered.
	\end{lemma}
	
	\begin{proof}
		Let $0\le k<2^{L_j}-1$. Since $p_{k,j}\le 2^{A_j-L_j-3}<2^{A_j+k}$, we get
		\[
			\max I_{k,j}=2^{A_j+k}+n_{(A_j,L_j)}+p_{k,j}<2^{A_j+k}+n_{(A_j,L_j)}+2^{A_j+k}=	2^{A_j+k+1}+n_{(A_j,L_j)}<\min I_{k+1,j}.
		\]
		Thus the intervals are disjoint for fixed $j$.
		
		The separation between different $j$-levels follows by choosing $A_{j+1}$ so large that
		the last endpoint at level $j$ is smaller than the first point at level $j+1$. This is possible because $A_{j+1}$ may be increased without breaking the conditions	required in Lemma \ref{bg14}.
	\end{proof}
	
	\begin{remark}\label{rem1}
		The additional requirements
		\[
			A_j\ge L_j+2^{L_j}+2,\qquad	P_{j-1}<p_{0,j}
		\]
		can be imposed when the sequences $A_j,L_j$ are chosen. Indeed, after the previous levels have been fixed, $P_{j-1}$ is already fixed. We may then choose $A_j$ so large that
		\[
			2^{A_j-L_j-2^{L_j}-2}>P_{j-1}.
		\]
		This recursive choice is compatible with Lemma \ref{bg14}: after $A_j$ has been	chosen, we choose $L_{j+1}$ large enough to satisfy the condition required there, and then choose $A_{j+1}$ sufficiently large to meet the additional requirements above.
	\end{remark}
	
	\begin{theorem}\label{main}
		There exist a strictly increasing sequence of positive integers $(a_n)$
		and a function $f\in L^1(G)$ such that
		\[
			\left(\frac1N\sum_{n=1}^N S_{a_n}f(x)\right)_{N=1}^{\infty}
		\]
		diverges for almost every $x\in G$.
	\end{theorem}
	
	\begin{proof}
		Choose the sequences $A_j,L_j$ as in Lemma \ref{bg14}, with the additional
		requirements
		\[
			A_j\ge L_j+2^{L_j}+2,\qquad	P_{j-1}<p_{0,j}.
		\]
		By Remark \ref{rem1} this is possible. Let
		\[
			f:=\sum_{j=1}^{\infty}\alpha_j Q_{A_j,L_j}.
		\]
		By Lemma \ref{A}, $f\in L^1(G)$.
		
		Let $(a_n)$ be the increasing enumeration of
		\[
			\bigcup_{j=1}^{\infty}\bigcup_{k=0}^{2^{L_j}-1} I_{k,j}.
		\]
		By Lemma \ref{C} this is a strictly increasing sequence of positive integers.
		
		Fix $x\in G\setminus E$, where $E$ is the null set obtained in Lemma \ref{A}. This exceptional set is chosen independently of $j$, $k(j,x)$, and of the subsequence $q_j$. Then the conclusion	of Lemma \ref{A} holds for this $x$ and for all sufficiently large $j$. We show that the arithmetic means cannot converge at this $x$.
		
		For $0\le k<2^{L_j}$ put
		\[
			N_{k,j}^{-}:=P_{j-1}+\sum_{\nu=0}^{k-1}p_{\nu,j},\qquad N_{k,j}:=N_{k,j}^{-}+p_{k,j}.
		\]
		So $N_{k,j}$ is the position of the last element of the block $I_{k,j}$ in the sequence $(a_n)$.
		
		For each large $j$ choose $k=k(j,x)$ as in Lemma \ref{A}. We apply Lemma \ref{A} with the dyadic length $q=p_{k,j}$. This is allowed because $p_{k,j}\le 2^{A_j-L_j-3}\le 2^{A_j-L_j-1}$. Moreover, $p_{k,j}\ge p_{0,j}$, and $p_{0,j}\to\infty$, since
		$P_{j-1}<p_{0,j}$ and $P_{j-1}\to\infty$. Thus the condition $q\to\infty$ in Lemma \ref{A} is also satisfied. Hence,
		\[
			\frac1{p_{k,j}}\sum_{\ell=2^{A_j+k}+n_{(A_j,L_j)}+1}^{2^{A_j+k}+n_{(A_j,L_j)}+p_{k,j}}S_\ell(f;x)\ge cj-o(j).
		\]
		After decreasing $c>0$ if needed, we may assume that for all large $j$,
		\[
			\frac1{p_{k,j}}\sum_{\ell=2^{A_j+k}+n_{(A_j,L_j)}+1}^{2^{A_j+k}+n_{(A_j,L_j)}+p_{k,j}}S_\ell(f;x)\ge cj.
		\]
		
		For all sufficiently large $j$, we also have $N_{k,j}^{-}>0$.	
		
		At the end of the block $I_{k,j}$, the prefix average is a weighted average of the
		previous prefix average and the present block average:
		\[
		\begin{aligned}
			\frac{1}{N_{k,j}}\sum_{n=1}^{N_{k,j}}S_{a_n}f(x)
			&=
			\frac{N_{k,j}^{-}}{N_{k,j}}\,
			\frac{1}{N_{k,j}^{-}}\sum_{n=1}^{N_{k,j}^{-}}S_{a_n}f(x)
			+
			\frac{p_{k,j}}{N_{k,j}}\,
			\frac1{p_{k,j}}
			\sum_{\ell\in I_{k,j}}S_\ell(f;x).
		\end{aligned}
		\]

		By Lemma \ref{B},
		\[
			\frac{p_{k,j}}{N_{k,j}}>\frac12,\qquad\frac{N_{k,j}^{-}}{N_{k,j}}<\frac12 .
		\]
		
		There are two cases.
		
		If
		\[
			\frac{1}{N_{k,j}^{-}}\sum_{n=1}^{N_{k,j}^{-}}S_{a_n}f(x)\le -\frac c2 j
		\]
		for infinitely many $j$, then the sequence of arithmetic means is unbounded below
		along a subsequence, hence it diverges.
		
		Otherwise, for all sufficiently large $j$,
		\[
			\frac{1}{N_{k,j}^{-}}\sum_{n=1}^{N_{k,j}^{-}}S_{a_n}f(x)>-\frac c2 j.
		\]
		For such $j$ we get
		\[
		\begin{aligned}
			\frac{1}{N_{k,j}}\sum_{n=1}^{N_{k,j}}S_{a_n}f(x)
			&>
			\frac{N_{k,j}^{-}}{N_{k,j}}\left(-\frac c2 j\right)
			+
			\frac{p_{k,j}}{N_{k,j}}cj
			>
			-\frac12\cdot\frac c2 j+\frac12 cj
			=
			\frac c4 j.
		\end{aligned}
		\]
		So $\limsup_{N\to\infty}\left|\frac1N\sum_{n=1}^{N}S_{a_n}f(x)\right|=+\infty$.
		
		In both cases the sequence of arithmetic means diverges at $x$.	Since the exceptional set in Lemma \ref{A} has measure zero, the divergence holds for almost every $x\in G$.
	\end{proof}


\begin{thebibliography}{99}
	
	\bibitem{BGcd} I. Blahota and G. G\'at,
	{\it Norm and almost everywhere convergence and divergence of matrix transform means of Walsh-Fourier series},
	J. Geom. Anal. \textbf{35}, 349 (2025).
	
	\bibitem{Belinsky1984} E.S. Belinsky,
	{\it On the summability of Fourier series with the method of lacunary arithmetic means},
	Anal. Math. \textbf{10} (1984), 275--282.
	
	\bibitem{Belinsky1997} E.S. Belinsky,
	{\it Summability of Fourier series with the method of lacunary arithmetical means at the Lebesgue points},
	Proc. Amer. Math. Soc. \textbf{125} (1997), no. 12, 3689--3693.
	
	\bibitem{Billard1967} P. Billard,
	{\it Sur la convergence presque partout des s\'eries de Fourier--Walsh des fonctions de l'espace $L^2(0,1)$},
	Studia Math. \textbf{28} (1967), 363--388.
	
	\bibitem{Carleson1966} L. Carleson,
	{\it On convergence and growth of partial sums of Fourier series},
	Acta Math. \textbf{116} (1966), 135--157.
	
	\bibitem{Carleson1983} L. Carleson,
	{\it Appendix to the paper by J.P. Kahane and Y. Katznelson, S\'eries de Fourier des fonctions born\'ees},
	in: Studies in Pure Mathematics, Birkh\"auser, Basel--Boston, Mass. (1983), 411--413.
	
	\bibitem{Fejer1904} L. Fej\'er,
	{\it Untersuchungen \"uber Fouriersche Reihen},
	Math. Ann. \textbf{58} (1904), 51--69.
	
	\bibitem{Fine} N.J. Fine,
	{\it Ces\'aro summability of Walsh-Fourier series},
	Proc. Natl. Acad. Sci. USA \textbf{41} (1955), no. 8, 588--591.
	
	\bibitem{Gat2010} G. G\'at,
	{\it Almost everywhere convergence of Fej\'er and logarithmic means of subsequences of partial sums of the Walsh--Fourier series of integrable functions},
	J. Approx. Theory \textbf{162} (2010), no. 4, 687--708.
	
	\bibitem{Gat2019} G. G\'at,
	{\it Ces\`aro means of subsequences of partial sums of trigonometric Fourier series},
	Constr. Approx. \textbf{49} (2019), no. 1, 59--101.
	
	\bibitem{Gat2019Walsh} G. G\'at,
	{\it On the convergence of Fej\'er means of some subsequences of partial sums of Walsh--Fourier series},
	Annales Univ. Sci. Budapest., Sect. Comp. \textbf{49} (2019), 187--198.
	
	\bibitem{GatTrigSubseqDivergence} G. G\'at,
	{\it Almost everywhere divergence of Ces\`aro means of subsequences of partial sums of trigonometric Fourier series}, Math. Ann. \textbf{389} (2024), 4199--4231.
	
	\bibitem{Gosselin1958} R.P. Gosselin,
	{\it On the divergence of Fourier series},
	Proc. Amer. Math. Soc. \textbf{9} (1958), 278--282.
	
	\bibitem{Hunt1968} R.A. Hunt,
	{\it On the convergence of Fourier series},
	in: Orthogonal Expansions and their Continuous Analogues,
	Southern Illinois University Press, Carbondale, Ill. (1968), 235--255.
	
	\bibitem{KahaneKatznelson1983} J.-P. Kahane and Y. Katznelson,
	{\it S\'eries de Fourier des fonctions born\'ees},
	in: Studies in Pure Mathematics, Birkh\"auser, Basel--Boston, Mass. (1983), 395--410.
	
	\bibitem{Kolmogorov1923} A.N. Kolmogoroff,
	{\it Une s\'erie de Fourier--Lebesgue divergente presque partout},
	Fund. Math. \textbf{4} (1923), 324--328.
	
	\bibitem{Kolmogorov1926} A.N. Kolmogoroff,
	{\it Une s\'erie de Fourier--Lebesgue divergente partout},
	C. R. Acad. Sci. Paris \textbf{183} (1926), 1327--1329.
	
	\bibitem{Konyagin1993} S.V. Konyagin,
	{\it The Fourier--Walsh subsequence of partial sums},
	Math. Notes \textbf{54} (1993), no. 3--4, 1026--1030.
	
	\bibitem{Lebesgue1905} H. Lebesgue,
	{\it Recherches sur la convergence des s\'eries de Fourier},
	Math. Ann. \textbf{61} (1905), 251--280.
	
	\bibitem{Paley} R.E.A.C. Paley,
	{\it A remarkable series of orthogonal functions (I)},
	Proc. London Math. Soc. \textbf{34} (1932), 241--264.
	
	\bibitem{Salem1955} R. Salem,
	{\it On strong summability of Fourier series},
	Amer. J. Math. \textbf{77} (1955), 393--403.
	
	\bibitem{SWSP} F. Schipp, W.R. Wade, P. Simon, and J. P\'al,
	{\it Walsh Series. An Introduction to Dyadic Harmonic Analysis},
	Adam Hilger, Bristol--New York (1990).
	

	\bibitem{Totik1982} V. Totik,
	{\it On the divergence of Fourier-series},
	Publ. Math. Debrecen \textbf{29} (1982), no. 3--4, 251--264.
	
	\bibitem{Zalcwasser1936} Z. Zalcwasser,
	{\it Sur la sommabilit\'e des s\'eries de Fourier},
	Stud. Math. \textbf{6} (1936), 82--88.
	

	
\end{thebibliography}
\end{document}